\documentclass[12pt,A4,reqno]{amsart}
\usepackage{amsfonts}
\usepackage{mathrsfs}
\usepackage[T1]{fontenc}
\usepackage{calligra}
\usepackage{amssymb}
\usepackage{amsmath,amscd}
\usepackage{color}
\usepackage{epsf}
\usepackage{graphicx}
\usepackage{xypic}

\theoremstyle{plain}
\newtheorem{thm}{Theorem}[section]
\newtheorem{lem}[thm]{Lemma}

\newtheorem{cor}[thm]{Corollary}

\theoremstyle{definition}
\newtheorem{defi}[thm]{Definition}
\newtheorem{rem}[thm]{Remark}

\newcommand{\Q}{\mathbb Q}
\newcommand{\R}{\mathbb R}
\newcommand{\Z}{\mathbb Z}

\newcommand{\nn}{\vskip 0.2cm}
\newcommand{\n}{\vskip 0.1cm}

\begin{document}

\title [\ ] {On transition functions of Topological Toric Manifolds}

\author{Li Yu}
\address{Department of Mathematics and IMS, Nanjing University, Nanjing, 210093, P.R.China
  \newline
     \textit{and}
 \newline
    \qquad  Department of Mathematics, Osaka City University, Sugimoto,
     Sumiyoshi-Ku, Osaka, 558-8585, Japan}

 \email{yuli@nju.edu.cn}



\thanks{
 The author is partially supported by
 the Japanese Society for
the Promotion of Sciences (JSPS grant no. P10018) and
 Natural Science Foundation of China (grant no.11001120).}


 \begin{abstract}
   We show that any topological toric manifold can be covered by finitely many
   open charts so that all the
   transition functions between these charts are Laurent monomials of
   $z_j$'s and $\overline{z}_j$'s.  In addition, we will describe
    toric manifolds and some special class of topological toric manifolds in terms of transition
   functions of charts up to (weakly) equivariant diffeomorphism.
  \end{abstract}

\maketitle

  \section{Background}

   Topological toric manifold was introduced as an analogue of toric manifolds
   (i.e. compact smooth toric varieties) in the category of closed smooth manifolds by H.~Ishida,
    Y.~Fukukawa and M.~Masuda, ~\cite{MasIshFuk2010}. By definition, a
   \textit{topological toric manifold} $X$ is a closed smooth manifold of dimension
   $2n$ with an effective smooth action of $(\mathbb{C}^*)^n$ so that
   $X$ has an open dense orbit, and $X$ is covered by finitely many
   $(\mathbb{C}^*)^n$-invariant open subsets each of which is equivariantly
   diffeomorphic to a smooth representation space of $(\mathbb{C}^*)^n$.
   Since the $(\mathbb{C}^*)^n$-action is effective, the smooth representation
   of $(\mathbb{C}^*)^n$ in each invariant open subset must be faithful, hence is isomorphic to
    a direct sum of complex one-dimensional smooth (linear) representation spaces of
    $(\mathbb{C}^*)^n$.  Notice that there is a
    \textit{canonical $(S^1)^n$-action} on $X$
    by restricting the $(\mathbb{C}^*)^n$-action to the standard
    compact torus $(S^1)^n \subset (\mathbb{C}^*)^n$. In this paper, we will always assume
    that a topological toric manifold $X$ is equipped with this
    $(S^1)^n$-action too. \nn

    It is shown in section 7 of~\cite{MasIshFuk2010} that a topological toric
    manifold $X$ is always simply connected.
    The orbit space of the $(S^1)^n$-action on $X$ is a nice manifold with
   corners whose faces (including the whole orbit space) are all
   contractible and, any intersection of faces is either connected
   or empty. So the orbit space looks like a simple polytope.
   In addition, similar to toric manifolds, the integral
   cohomology ring $H^*(X)$ is generated by the elements of
   $H^2(X)$.\nn

   In a $2n$-dimensional topological toric manifold $X$, there exist finitely many
   codimension-two closed connected submanifolds $X_1,\cdots, X_m$ in $X$
   each of which is fixed pointwise by some $\mathbb{C}^*$-subgroup of $(\mathbb{C}^*)^n$.
   In fact, the Poincar\'e dual to $X_1,\cdots, X_m$ form an integral basis
   of $H^2(X)$. Such $X_1,\cdots, X_m$ are called the \textit{characteristic submanifolds} of $X$.
    A choice of an orientation on each characteristic submanifold $X_i$ together with an orientation on $X$ is
    called an \textit{ominiorientation} on $X$. From $X_1,\cdots, X_m$, we can define an abstract simplicial
    complex $\Sigma(X)$ of dimension $n-1$ (with the empty set $\varnothing$ added) as following.
    \begin{equation}\label{Equ:Fan}
      \Sigma(X) := \{ I \subset [m]\, |\, X_I := \bigcap_{i\in I} X_i  \neq \varnothing \}
      \cup \{ \varnothing\}
    \end{equation}
   where $[m]$ denote the set $\{ 1,\cdots, m\}$.
   $\Sigma(X)$ is a \textit{pure} simplicial complex in the sense that any
   simplex in $\Sigma(X)$ is contained in some simplex of maximal
   dimension $(n-1)$.\nn

    Let
   $\Sigma^{(k)}(X)$ be the set of $(k-1)$-simplices in
   $\Sigma(X)$. Then $\Sigma^{(1)}(X)$ can be identified with $[m]$.
   It is shown by Lemma 3.6 in~\cite{MasIshFuk2010} that $X_1,\cdots, X_m$ intersect
   transversely with each other and any $X_I$ ($I\subset [m]$) is a
   closed connected submanifold of dimension $2(n-|I|)$ in $X$.
   The characteristic submanifolds of $X$
   play a similar role as invariant irreducible divisors
   of a toric variety.\nn

    The major difference between topological toric manifolds and toric manifolds is that
    the local action of $(\mathbb{C}^*)^n$ is equivariantly diffeomorphic to a smooth vs. algebraic
    linear representation space of $(\mathbb{C}^*)^n$.
    For example when $n=1$, any smooth complex one-dimensional representation
     of $\mathbb{C}^* = \R_{>0} \times S^1$
    corresponds to a smooth endomorphism of $\mathbb{C}^*$, which can be written in the following form:
   \begin{equation} \label{Equ:Smooth-Repre}
     g\ \longmapsto\ |g|^{b+\sqrt{-1} c} (\frac{g}{|g|})^v,\ \text{where}\ (b+\sqrt{-1} c , v)
      \in \mathbb{C}\times \Z .
    \end{equation}
  A smooth endomorphism of $\mathbb{C}^*$ in~\eqref{Equ:Smooth-Repre} is
  algebraic if and only if $b=v$ and
   $c=0$. So the group
   $\mathrm{Hom}(\mathbb{C}^*,\mathbb{C}^*)$ of \textit{smooth endomorphisms} of
   $\mathbb{C}^*$ is isomorphic to $\mathbb{C}\times \Z$, while the
   group   $\mathrm{Hom}_{alg}(\mathbb{C}^*,\mathbb{C}^*)$ of
   \textit{algebraic endomorphisms} of $\mathbb{C}^*$ is isomorphic to $\Z$.
  Hence there are much more smooth linear representations of
  $(\mathbb{C}^*)^n$ than algebraic ones. And so the family of topological toric manifolds
  is much larger than the family of toric manifolds.\nn

    There is another topological analogue of toric manifold introduced by Davis and
    Januszkiewicz~\cite{DaJan91} in the early 1990s, now called
    ``quasitoric manifold'' (see~\cite{BP02}).
    A \textit{quasitoric manifold} is an $2n$-dimensional smooth manifold
    with a locally standard $(S^1)^n$-action whose orbit space is a simple convex
    polytope. It is shown in section 10 of~\cite{MasIshFuk2010} that any quasitoric
     manifold has a structure of topological toric manifold. In fact, there are uncountably many
      topological toric manifold structure for any given quasitoric
      manifold $M$. But there is no canonical choice of such a structure for $M$.\nn

    It is shown in section 10 of~\cite{MasIshFuk2010} that the
     family of $2n$-dimensional topological toric manifolds with the canonical
     $(S^1)^n$-actions is strictly larger than the family of $2n$-dimensional quasitoric
     manifolds up to equivariant homeomorphisms. The relations between toric manifolds,
      quasitoric manifolds and
      topological toric manifolds can be explained by the
      diagram in Figure~\ref{p:Range-Map} (see chapter 5 of~\cite{BP02}). \nn

   \begin{figure}
  \begin{equation*}
   \vcenter{
            \hbox{
                  \mbox{$\includegraphics[width=0.8\textwidth]{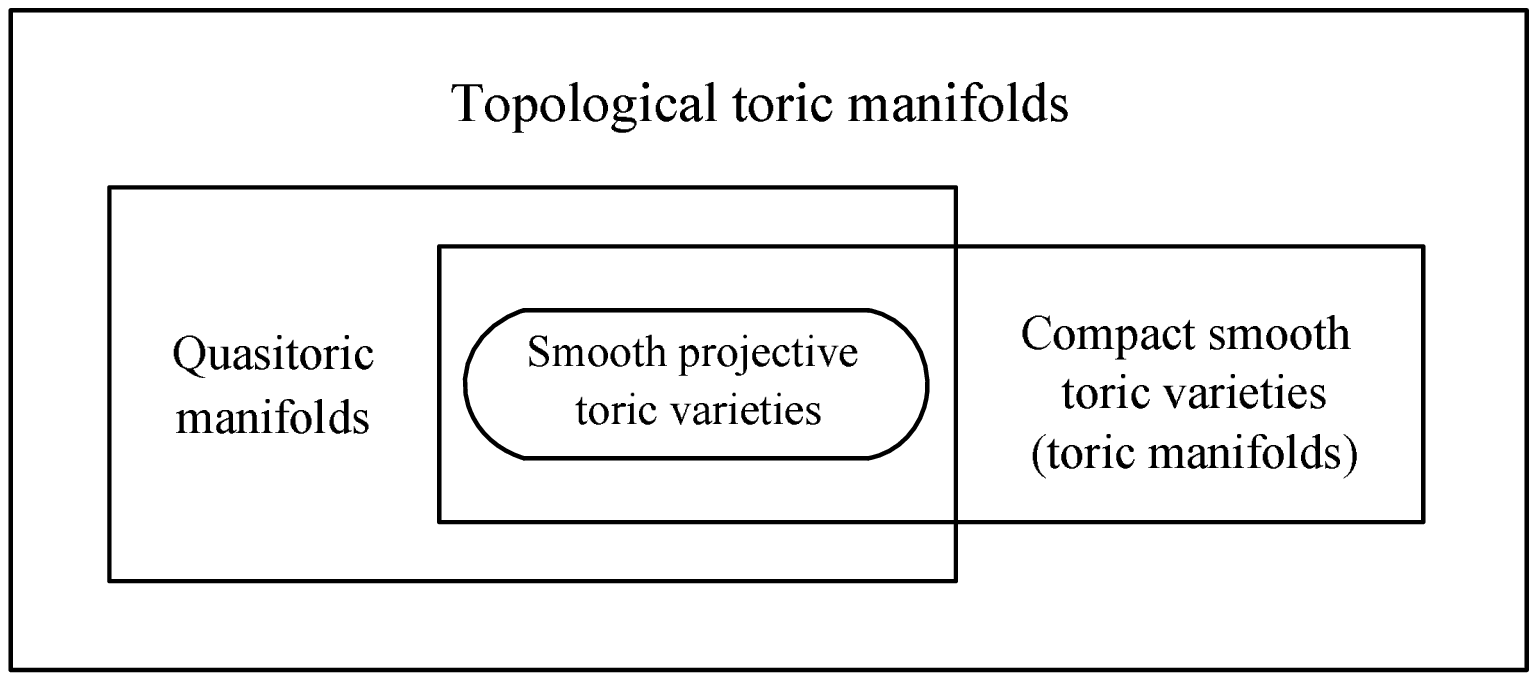}$}
                 }
           }
  \end{equation*}
  \caption{
      \label{p:Range-Map}  }
 \end{figure}

      In this paper, we will study topological toric manifolds from the viewpoint
       of transition functions of charts.
       The paper is organized as follows. In section 2, we will review
    the basic construction in~\cite{MasIshFuk2010} of a topological toric manifold $X(\Delta)$ from a
    complete non-singular topological fan $\Delta$, and define a set
    of open charts called \textit{normal charts} which cover the whole $X(\Delta)$.
    A very special property of these charts is that the transition
    functions between any two normal charts are
    completely determined by the topological fan $\Delta$. This
    allows us to give an equivalent description of toric manifolds in terms of transition
   functions of some charts up to (weakly) equivariant diffeomorphism (see
   Theorem~\ref{thm:Toric-Trans-Func}).
    In section 3, we will show that any topological toric manifold
    can be covered by finitely many $(\mathbb{C}^*)^n$-invariant open charts
     so that all the transition functions are Laurent monomials of
      $z_j$'s and $\overline{z}_j$'s (see Corollary~\ref{Cor:Trans-Func}).
        There should be some special geometrical properties on a topological toric manifold
   implied by the existence of such kind of atlas. It is interesting to
    see what these geometric properties are.
    In section 4, we will
   define the notion of \textit{nice topological toric manifold}.
   Similar to toric manifolds, we can also
    describe nice topological toric manifolds in terms of transition
   functions of charts up to (weakly) equivariant diffeomorphism (see Theorem~\ref{thm:Top-Toric-Trans-Func}).\nn

    In this paper, we will quote many lemmas and theorems from~\cite{MasIshFuk2010}.
    But we will avoid repeat of
    the full statements of those lemmas and theorems. Instead, we will only
      indicate where they are in~\cite{MasIshFuk2010}. So the reader should
     get familiar with the content of~\cite{MasIshFuk2010} before reading the arguments in
     this paper.\\

   \section{Topological fan, transition functions and Normal charts}

   Suppose a $2n$-dimensional topological toric manifold
   $X$ is equipped with an ominiorientation. Using the notations in the discussion in section 1,
    for each characteristic submanifold $X_i$ in $X$, $i\in [m] = \Sigma^{(1)}(X)$,
      the $\mathbb{C}^*$-subgroup of $(\mathbb{C}^*)^n$ which
   fixes $X_i$ determines a unique element
   $\lambda_{\beta_i(X)} \in \mathrm{Hom}(\mathbb{C}^*,
   (\mathbb{C}^*)^n)$ where $\beta_i(X) \in \mathbb{C}^n\times \Z^n$ (see Lemma 3.3 in~\cite{MasIshFuk2010}).
   So we have a map
    $$\beta(X): \Sigma^{(1)}(X) = [m] \rightarrow \mathbb{C}^n\times
    \Z^n, \ \, \beta(X)(i) = \beta_i(X) \ \text{for}\ \forall \, i\in [m]. $$

    It is shown by Lemma 3.8 in~\cite{MasIshFuk2010} that $\Delta(X) := (\Sigma(X), \beta(X))$
   is a complete non-singular \textit{topological fan} of dimension
   $n$. Roughly speaking, a topological fan consists of
   two simplicial fans which encode the information of the $\mathbb{C}^n$ and $\Z^n$
   components of $\beta(X)$, respectively (see section 3 of~\cite{MasIshFuk2010}
   for the precise definition).\nn

   Conversely, Theorem 8.1 in~\cite{MasIshFuk2010} shows that any
   complete non-singular topological fan $\Delta = (\Sigma, \beta)$ of dimension $n$ determines
   an ominioriented topological toric manifold $X(\Delta)$ of dimension $2n$ up
   to ominiorientation-preserving equivariant
   diffeomorphism. Indeed, $X(\Delta)$ is defined to be the following quotient space
      \begin{equation}\label{Equ:Cox-Constr}
      X(\Delta) := U(\Sigma) \slash \mathrm{Ker}(\lambda_{\beta})
      \end{equation}
  where $U(\Sigma)$ is an open subspace of $\mathbb{C}^m$
  which depends only on $\Sigma$ and $m=|\Sigma^{(1)}|$ is the number of $0$-simplices in $\Sigma$.
  More precisely,
    \begin{equation} \label{Equ:U}
      U(\Sigma):= \bigcup_{I\in \Sigma} U(I)
  \end{equation}
    $$U(I):= \{ (z_1,\cdots, z_m )\in \mathbb{C}^m\, |\, z_i\neq 0\
     \text{for}\ \forall\, i\notin I \}.$$
  \n

  In~\eqref{Equ:Cox-Constr}, $\lambda_{\beta} : (\mathbb{C}^*)^m \rightarrow (\mathbb{C}^*)^n$
  is a surjective group homomorphism determined by $\beta$. The natural action of $(\mathbb{C}^*)^m$
  on $\mathbb{C}^m$ leaves $U(\Sigma)$ invariant, so it induces an effective action of
  $(\mathbb{C}^*)^m\slash \mathrm{Ker}(\lambda_{\beta})$ on $X(\Delta)$ which has an open dense orbit.
  Since $\lambda_{\beta}$ is surjective,  $(\mathbb{C}^*)^m\slash \mathrm{Ker}(\lambda_{\beta})$
  is isomorphic to $(\mathbb{C}^*)^n$. In this way, $X(\Delta)$ get a $(\mathbb{C}^*)^n$-action
  with an open dense orbit.
  Moreover, if we restrict $\lambda_{\beta}$ to the standard compact torus
  $(S^1)^m \subset (\mathbb{C}^*)^m$, we get a surjective group homomorphism
   \begin{equation} \label{Equ:lambda-real}
   \widehat{\lambda}_{\beta} : (S^1)^m  \rightarrow (S^1)^n \subset
    (\mathbb{C}^*)^n.
  \end{equation} \nn
  If we write $\beta_i = (b_i+ \sqrt{-1} c_i, v_i) \in  \mathbb{C}^n\times \Z^n$,
  $i\in [m] = \Sigma^{(1)}$,
    then $\widehat{\lambda}_{\beta}$ is determined only by the component $\{ v_i\}_{i\in [m]}$ (see section 4
    of~\cite{MasIshFuk2010}).  It is clear that
    \begin{equation} \label{Equ:lambda-real-2}
         \mathrm{Ker}(\widehat{\lambda}_{\beta}) =
         \mathrm{Ker}(\lambda_{\beta}) \cap (S^1)^m \subset (\mathbb{C}^*)^m.
    \end{equation} \nn

  The correspondence between ominioriented topological toric manifolds
   and complete non-singular topological fans generalizes the
   classical correspondence between toric manifolds and complete non-singular
   fans (see~\cite{Danilov} or~\cite{Oda}).

   \nn

  \begin{defi}[Weakly Equivariant Diffeomorphism]
   For a Lie group $G$, two smooth $G$-manifolds $M$ and $N$ are called
    \textit{weakly equivariantly diffeomorphic via a
    diffeomorphism}
     $F : M \rightarrow N$ if there exists an automorphism $\sigma$ of $G$ so that
    $$ F(g\cdot x) = \sigma(g)\cdot F(x) \ \, \text{for any}\ g\in G\ \text{and any}\ x \in M.$$
     And we call $F$ a \textit{weakly equivariant diffeomorphism} between $M$ and $N$.
    If $\sigma$ can be taken to be the identity map of $G$, then
   $M$ and $N$ are called \textit{equivariantly diffeomorphic} and
   $F$ is called an \textit{equivariant diffeomorphism}.\n
    Similarly, we can define the notation of (weakly) equivariant
   homeomorphism by replacing ``diffeomorphism'' by
   ``homeomorphism'' in the above definition.
  \end{defi}

   \n

  It is shown in section 3 of~\cite{MasIshFuk2010} that
  there are three different levels of equivalence
   relations among complete non-singular topological fans which give isomorphism, equivariant diffeomorphism
    and equivariant homeomorphism among the corresponding ominioriented topological toric
    manifolds,
   respectively (see Lemma 3.9 in~\cite{MasIshFuk2010}). \nn

   From any complete non-singular topological fan $\Delta = (\Sigma, \beta)$ of dimension $n$, we can
     define an atlas $\mathcal{U}$ on the ominioriented topological toric manifold $X(\Delta)$
      whose charts are indexed by all the $(n-1)$-simplices $I \in \Sigma^{(n)}$.
   Indeed, the open chart in $\mathcal{U}$ corresponding to an $I\in \Sigma^{(n)}$ is
   \begin{equation} \label{Equ:Charts}
    \varphi_I : V_I = U(I) \slash \mathrm{Ker}(\lambda_{\beta}) \rightarrow \mathbb{C}^n.
   \end{equation}
    It is easy to see that each $V_I \subset X(\Delta)$ is \textit{invariant} under
     the $(\mathbb{C}^*)^n$-action on $X(\Delta)$, i.e. $g(V_I) \subset V_I$ for any $g\in (\mathbb{C}^*)^n$
      (see section 4 of~\cite{MasIshFuk2010}).
     We call each $\varphi_I : V_I \rightarrow \mathbb{C}^n$ a \textit{normal chart} of
     $X(\Delta)$.
      It is clear that $X(\Delta)$ is covered by finitely many normal charts
     $\{\varphi_I : V_I \rightarrow \mathbb{C}^n,\, I\in\Sigma^{(n)} \}$.
     When we observe the $(\mathbb{C}^*)^n$-action on $X(\Delta)$ through any
      normal chart $\varphi_I : V_I \rightarrow \mathbb{C}^n$,
     the $(\mathbb{C}^*)^n$-action is a faithful smooth linear representation.
     In other words, the $(\mathbb{C}^*)^n$-action in $V_I$ is equivariantly homeomorphic to
     a faithful smooth linear representation of $(\mathbb{C}^*)^n$ in $\mathbb{C}^n$ via
     the map $\varphi_I$.\nn

   Next, let us see what the characteristic submanifolds of $X(\Delta)$ are.
   We define
   \begin{equation} \label{Equ:Charts-2}
    U(\Sigma)_i := U(\Sigma) \cap \{ (z_1,\cdots, z_m) \in \mathbb{C}^m \, | \, z_i
    =0 \}, \,\ i \in \Sigma^{(1)}.
   \end{equation}

    Let $\Phi_{\Delta} : U(\Sigma) \rightarrow U(\Sigma)\slash \mathrm{Ker}(\lambda_{\beta})=X(\Delta)$
    denote the quotient map. Then all the characteristic submanifolds of
    $X(\Delta)$ are
     $$X(\Delta)_i = \Phi_{\Delta}(U(\Sigma)_i),\ \, i\in \Sigma^{(1)}.$$
     \nn

    Observe that for any $I\in \Sigma^{(n)}$, we have
     \begin{equation} \label{Equ:U(I)}
     U(I)= U(\Sigma) - \bigcup_{i\notin I}
     U(\Sigma)_i,
    \end{equation}
      \begin{equation} \label{Equ:V(I)}
    \text{and so}\quad  V_I =  \Phi_{\Delta}(U(I)) = X(\Delta) - \bigcup_{i\notin I}
    X(\Delta)_i. \qquad\ \
      \end{equation}
   So each open subset $V_I$ is the complement of
  a set of characteristic submanifolds $\{ X(\Delta)_i, i\notin I \}$ in $X(\Delta)$.
  \nn

   \begin{rem} \label{Rem:Char-SubMfd}
    If we consider $X(\Delta)$ as a $(S^1)^n$-manifold,
    we can similarly define the notion of characteristic
    submanifold of $X(\Delta)$ with respect to the
    $(S^1)^n$-action, which are submanifolds
    fixed pointwise by some $S^1$-subgroup of $(S^1)^n$. Indeed, this notion has been used
    in the study of quasitoric manifolds in~\cite{DaJan91} and~\cite{BP02}.
    It is easy to see that the characteristic submanifolds of
    $X(\Delta)$ with respect to the $(\mathbb{C}^*)^n$-action and
    the canonical $(S^1)^n$-action actually coincide.
   \end{rem}

     The transition functions between any two normal charts $\varphi_I: V_I \rightarrow \mathbb{C}^n$
    and $\varphi_J: V_J \rightarrow \mathbb{C}^n$
    are computed in Lemma 5.2 in~\cite{MasIshFuk2010} via the information
    of $\beta$. Let
     \begin{equation} \label{Equ:Beta}
       \beta_i = (b_i+ \sqrt{-1} c_i, v_i) \in   \mathbb{C}^n\times \Z^n, \
      i\in \Sigma^{(1)}.
     \end{equation}
     Here we consider $b_i,c_i \in \R^n$ and $v_i\in \Z^n$ as column vectors.
      The following are some easy consequences of
    Lemma 5.2 in~\cite{MasIshFuk2010}.\nn

   $\bullet$\   If $b_i=v_i$ and $c_i=0$ for all $i\in  \Sigma^{(1)}$, then
      $X(\Delta)$ is a toric manifold. In this case, the transition
      functions between any two charts in $\mathcal{U}$ is of the form
       $$\omega_j = f_j(z_1,\ldots, z_n), \ 1\leq j \leq n$$
      where each $f_j(z_1,\ldots, z_n)$ is a Laurent monomial in complex
      variables $z_1,\ldots, z_n$.\nn

      Conversely, Suppose $M^{2n}$ is a $2n$-dimensional closed
     smooth manifold with an effective smooth action of $(\mathbb{C}^*)^n$ having an open dense
       orbit. And we assume\n

       \begin{enumerate}
       \item[(i)] $M^{2n}$ is covered by finitely many $(\mathbb{C}^*)^n$-invariant open
       charts $\{ \phi_j : V_j \rightarrow \mathbb{C}^n\}_{1\leq j\leq r}$
       with all the transition functions being Laurent monomials of $z_1,\cdots, z_n$,\n

      \item[(ii)] In one chart $\phi_j : V_j \rightarrow \mathbb{C}^n$,
      the $(\mathbb{C}^*)^n$-action in $V_j$ is (weakly) equivariantly homeomorphic to
       an algebraic linear representation of $(\mathbb{C}^*)^n$ in $\mathbb{C}^n$
       via the map $\phi_j$.
      \end{enumerate}

      Then since the transition functions of $\{ \phi_j : V_j \rightarrow \mathbb{C}^n\}_{1\leq j\leq r}$
      are all Laurent monomials in
     $z_1,\ldots, z_n$, it is easy to see that for any other chart
     $\phi_{j'} : V_{j'} \rightarrow \mathbb{C}^n$,
       the $(\mathbb{C}^*)^n$-action in $V_{j'}$ is also (weakly) equivariantly homeomorphic to
       an algebraic linear representation of $(\mathbb{C}^*)^n$ in $\mathbb{C}^n$
       via the map $\phi_{j'}$.
      Hence
       $M^{2n}$ is (weakly) equivariantly diffeomorphic to a
       toric manifold.
        So we can describe toric manifolds up to (weakly) equivariant diffeomorphism in terms of
       transition functions of charts as follows.\nn

     \begin{thm} \label{thm:Toric-Trans-Func}
       Suppose $M^{2n}$ is a $2n$-dimensional closed smooth manifold with an
       effective smooth action of $(\mathbb{C}^*)^n$ having an open dense orbit.
       Then $M^{2n}$ is (weakly) equivariantly diffeomorphic
       to a toric manifold if and only if
       $M^{2n}$ can be covered by finitely many $(\mathbb{C}^*)^n$-invariant
       open charts $\{ \phi_j : V_j \rightarrow \mathbb{C}^n\}_{1\leq j\leq r}$
       so that all the transition functions between these charts are Laurent monomials of
        $z_1,\ldots, z_n$ and, for at least one chart
       $\phi_j : V_j \rightarrow \mathbb{C}^n$,
         the $(\mathbb{C}^*)^n$-action in $V_j$ is (weakly) equivariantly homeomorphic to
       an algebraic linear representation of $(\mathbb{C}^*)^n$ in $\mathbb{C}^n$
       via the map $\phi_j$.
   \end{thm}
   \n

    $\bullet$\ If $b_i$ is an integral vector congruent to $v_i$
    modulo $2$ and $c_i=0$ for all $i\in \Sigma^{(1)}$, then the transition
    functions between any two charts in $\mathcal{U}$ has the form
       $$\omega_j = f_j(z_1,\ldots, z_n,\overline{z}_1,\ldots, \overline{z}_n), \ 1\leq j\leq n$$
        where each $f_j(z_1,\ldots, z_n,\overline{z}_1,\ldots, \overline{z}_n)$ is a Laurent monomial
        of $z_1,\ldots, z_n, \overline{z}_1,\ldots, \overline{z}_n$,
        $$\text{i.e.}\quad \ f_j(z_1,\ldots, z_n,\overline{z}_1,\ldots,
        \overline{z}_n) = \prod^n_{i=1}z^{p_{ij}}_i
        \overline{z}^{q_{ij}}_i, \ \text{where $ p_{ij}, q_{ij} \in \Z$}.$$
 \nn

  \begin{defi}[Nice Topological Fan]
    A complete non-singular topological fan $\Delta=(\Sigma,\beta)$ of dimension $n$ is called \textit{nice}
    if for each $i\in \Sigma^{(1)}$,
     we have
      $$ \text{$\beta_i = (b_i+ \sqrt{-1} \cdot 0, v_i) \in   \mathbb{C}^n\times \Z^n$
     where $b_i \in \Z^n$ and $b_i\equiv v_i \mod 2$}. $$
     By the above discussions, for
     any nice topological fan $\Delta=(\Sigma,\beta)$ of dimension $n$, the corresponding
     topological toric manifold $X(\Delta)$
     can be covered by finitely many $(\mathbb{C}^*)^n$-invariant open charts
     $ \{ \varphi_I : V_I \rightarrow \mathbb{C}^n \}_{I\in \Sigma^{(n)}}$
      with all the transition functions being Laurent monomials of
   $z_1,\ldots, z_n, \overline{z}_1,\ldots, \overline{z}_n$.
   \end{defi}

   In section 5 of~\cite{MasIshFuk2010}, an explicit ominioriented topological toric manifold structure
   on $\mathbb{C} P^2 \# \mathbb{C} P^2$ is constructed. The corresponding topological fan is nice.
    But $\mathbb{C} P^2 \# \mathbb{C} P^2$ is not a toric manifold.
    So the family of topological toric manifolds is strictly bigger
    than the family of toric manifolds. In addition, it is natural
    to ask the following question.
    \nn

 \noindent  \textbf{Question:} Among all ominioriented topological toric manifolds,
  how many of them have nice topological fans?\nn

   An answer to this question will be given by
   Theorem~\ref{thm:Fan-Deform-Nice} and
  Theorem~\ref{thm:Nice-Fan} in the next section.\\

  \section{Transition functions of Topological toric manifolds}

  According to the discussion in section 2, an ominioriented topological toric
  manifold with a nice topological fan can be covered by finitely many open
  charts with all the transition functions
  being Laurent monomials of $z_1,\ldots, z_n, \overline{z}_1,\ldots, \overline{z}_n$.
  In this section, we will show that the topological fan of any
  ominioriented topological toric manifold can be turned into a nice topological fan
   via a regular deformation defined as following.
  \nn

   \begin{defi} [Regular Deformation]
     Suppose $\Sigma$ is an $(n-1)$-dimensional (abstract) pure simplicial complex and
  $\Delta(t) = (\Sigma, \beta(t))$, $0\leq t \leq 1$, is a
  family of complete non-singular topological fans of dimension $n$
  which depend on the parameter $t$ smoothly. Then we call $\Delta(t)$
   a \textit{regular deformation} of $\Delta(0)$. Let
    \begin{equation} \label{Equ:Beta(t)}
       \beta_i(t) = (b_i(t)+ \sqrt{-1} c_i(t), v_i(t)) \in   \mathbb{C}^n\times \Z^n, \
      i\in \Sigma^{(1)}.
     \end{equation}
   It is clear that $v_i(t) = v_i(0) \in \Z^n$ for each $ i\in \Sigma^{(1)}$ since
   $\Z^n$ is a discrete set. So each $v_i(t)$ is independent on the parameter $t$.
    Hence for any complete non-singular topological fan $\Delta = (\Sigma, \beta)$,
    there always exists some other complete non-singular topological
    fan $\Delta'= (\Sigma, \beta')$ so that $\Delta$ can not be turned into $\Delta'$
       via any regular deformation.
  \end{defi}

  The following lemma is the key to the main results of this paper.\nn

 \begin{lem} \label{Lem:Deformation}
  Suppose
  $\Delta(t) = (\Sigma, \beta(t))$, $0\leq t \leq 1$ is
  a regular deformation of complete non-singular topological fans of dimension $n$. Then there exists a diffeomorphism
  $\psi : X(\Delta(0)) \rightarrow X(\Delta(1))$ with the following properties.
  \begin{enumerate}
     \item[(i)] $\psi$ is equivariant with respect to the canonical $(S^1)^n$-action
       on $X(\Delta(0))$ and $X(\Delta(1))$.\n

     \item[(ii)] $\psi$ maps each open subset
     $V_I(0) = U(I) \times \{ 0 \} \slash \mathrm{Ker}(\lambda_{\beta(0)})$
         in $X(\Delta(0))$ to $V_I(1) = U(I) \times \{ 1 \} \slash \mathrm{Ker}(\lambda_{\beta(1)})$
          in $X(\Delta(1))$ for any $I\in \Sigma^{(n)}$.
  \end{enumerate}
 \end{lem}
 \begin{proof}
   By definition, $X(\Delta(t)) = U(\Sigma) \slash \mathrm{Ker}(\lambda_{\beta(t)})$,
    $0\leq t \leq 1$. Let $m = |\Sigma^{(1)}|$ be the number of
    $0$-simplices in $\Sigma$.
    Then we have a smooth map $\Lambda$ defined by
      \begin{align*}
    \Lambda: (\mathbb{C}^*)^m \times [0,1] &\longrightarrow \, (\mathbb{C}^*)^n\\
                 (\, g\, ,\, t\, )\ &\longmapsto \lambda_{\beta(t)}(g)
    \end{align*}
      $$ \text{Let}\quad \Lambda_t^{-1}(e) : =\Lambda^{-1}(e) \cap ((\mathbb{C}^*)^m \times \{ t
   \}), \ t \in [0,1], $$
   where $e$ is the unit element of $(\mathbb{C}^*)^n$.
    Obviously, $ \Lambda_t^{-1}(e) \cong \mathrm{Ker}(\lambda_{\beta(t)})$, $t\in [0,1]$.
     \nn

   Similarly, for $ (S^1)^m \times [0,1] \subset (\mathbb{C}^*)^m \times [0,1]$,
   we have a smooth map
      \begin{align*}
    \widehat{\Lambda}: (S^1)^m \times [0,1] &\longrightarrow \, (S^1)^n \subset (\mathbb{C}^*)^n\\
                 (\, g\, ,\, t\, )\ &\longmapsto
                 \widehat{\lambda}_{\beta(t)}(g) \ \, \text{(see~\eqref{Equ:lambda-real})}
    \end{align*}
     $$ \text{Let}\ \ \widehat{\Lambda}_t^{-1}(e)
     : =\widehat{\Lambda}^{-1}(e) \cap ((S^1)^m \times \{ t \}) \cong
      \mathrm{Ker}(\widehat{\lambda}_{\beta(t)}), \ t \in
     [0,1]. $$\nn

      Suppose $\beta_i(t) = (b_i(t)+ \sqrt{-1} c_i(t), v_i) \in \mathbb{C}^n\times \Z^n, \  i\in
  [m]=\Sigma^{(1)}$, where $v_i \in \Z^n$ is a fixed vector for each
  $i\in [m]$. Then since $\widehat{\lambda}_{\beta(t)}$ is determined only by $\{ v_i \}_{i\in
  [m]}$, so $\widehat{\lambda}_{\beta(t)}$ is independent on $t$.
  Hence we have
   \begin{equation} \label{Equ:Bundle-Trivial}
     \widehat{\Lambda}^{-1}(e) =  \widehat{\Lambda}_0^{-1}(e) \times
   [0,1] \subset (S^1)^m \times [0,1].
   \end{equation}

   Next, we introduce an
   equivalence relation between any two points $(x,t)$ and $(x',t')$ in
    $\mathbb{C}^m \times [0,1]$ by
   \[ (x,t) \sim (x',t')\, \Longleftrightarrow\,  t'=t,\
    x'= g\cdot x \ \text{for some}\ g\in  \Lambda_t^{-1}(e)\]
   Let the quotient space of $U(\Sigma) \times [0,1]$ with respect to
   $\sim$ be denoted by
   $$U(\Sigma) \times [0,1] \slash \Lambda^{-1}(e).$$
   We use $[(x,t)]$ to denote the equivalence class of $(x,t)$ in
   $U(\Sigma) \times [0,1] \slash \Lambda^{-1}(e)$.
  We remark that $\Lambda^{-1}(e) \subset (\mathbb{C}^*)^m \times [0,1]$ is not a
   group although each $\Lambda_t^{-1}(e)$ is.\nn

   \textbf{Claim:} $U(\Sigma) \times [0,1] \slash  \Lambda^{-1}(e)$
      is a smooth manifold (with boundary).\nn

   First, we need to show that
   $U(\Sigma) \times [0,1] \slash  \Lambda^{-1}(e)$ with the
   quotient topology is Hausdorff. Let
   $[(x,t)]$ and $[(x',t')]$ be two different points in $U(\Sigma) \times [0,1] \slash
   \Lambda^{-1}(e)$.
   \begin{itemize}
     \item[(a)] if $t \neq t'$, then there exists a small real number $\varepsilon >0$ so that
     $(t-\varepsilon, t+\varepsilon)$ and $(t'-\varepsilon,
     t'+\varepsilon)$ are disjoint subintervals of $[0,1]$. Then
     $[(x,t)]$ and $[(x',t')]$ are
     contained in disjoint open subsets $U(\Sigma) \times (t-\varepsilon, t+\varepsilon) \slash
   \Lambda^{-1}(e)$ and $U(\Sigma) \times (t'-\varepsilon, t'+\varepsilon) \slash
   \Lambda^{-1}(e)$, respectively.\n

     \item[(b)] if $t=t'$ and $x\neq x'$, then  $[(x,t)]$ and $[(x',t')]=[x',t]$ are both
     contained in
     $U(\Sigma)\times \{ t\}
   \slash  \Lambda^{-1}(e) = U(\Sigma)\times \{ t\}
   \slash  \Lambda_t^{-1}(e) \cong U(\Sigma) \slash \mathrm{Ker}(\lambda_{\beta(t)}) =
    X(\Delta(t))$. Since $X(\Delta(t))$ is Hausdorff (see Lemma 6.1
    in~\cite{MasIshFuk2010}), so there exist two disjoint open subsets $W$ and $W'$ of $U(\Sigma)$
    with
  $[(x,t)] \in W\times \{ t\} \slash  \Lambda^{-1}_t(e)$ and $[(x',t)] \in W'\times \{ t \}\slash
     \Lambda^{-1}_t(e)$.
    Then $[(x,t)]$ and $[(x',t')]$ are
     contained in disjoint open subsets $W\times [0,1] \slash
   \Lambda^{-1}(e)$ and $W'\times [0,1] \slash \Lambda^{-1}(e)$, respectively.
      \end{itemize}
   So in any case, $[(x,t)]$ and $[(x',t')]$ can be separated by disjoint
   open subsets of $U(\Sigma) \times [0,1] \slash \Lambda^{-1}(e)$. This
   means that $U(\Sigma) \times [0,1] \slash  \Lambda^{-1}(e)$ is Hausdorff.   \n

    Second, we need to define a smooth structure on $U(\Sigma) \times [0,1] \slash  \Lambda^{-1}(e)$.
   Notice that we can cover $U(\Sigma) \times [0,1] \slash  \Lambda^{-1}(e)$ by
   finitely many open subsets
   $$\{ U(I)\times [0,1] \slash \Lambda^{-1}(e) \, ;\,I\in \Sigma^{(n)}
   \}.$$
   For any $I\in \Sigma^{(n)}$, let $\mathbb{C}^I$ be the affine
   space $\mathbb{C}^n$ with coordinates indexed by the elements of
   $I$. For any $t\in [0,1]$, let $\{ \alpha^{I}_i (t)\}_{i\in I}$ be the dual set of
   $\{ \beta_i (t) \}_{i\in I}$.
    Then similar to (5.1) in~\cite{MasIshFuk2010}, we define
   \begin{equation} \label{Equ:Charts-2}
    \overline{\varphi}_I: U(I)\times [0,1] \slash \Lambda^{-1}(e) \rightarrow
    \mathbb{C}^I\times [0,1]
   \end{equation}
   \[ \text{by}\quad \overline{\varphi}_I \left([(z_1,\ldots, z_m,t)]\right) =
   \left((\prod^m_{k=1} z^{\langle \alpha^I_{i}(t), \beta_k(t)\rangle}_k)_{i\in
   I},t \large\right)
     = \left((z_i \prod_{k\notin I} z^{\langle \alpha^I_{i}(t), \beta_k(t)\rangle}_k)_{i\in I},t \right)\]
     where $(z_1,\ldots, z_m) \in U(I) \subset \mathbb{C}^m$ and $t\in [0,1]$.
    The definitions of dual set and the pairing $\langle \, ,\, \rangle$
    are given in section 2 of~\cite{MasIshFuk2010}.\nn

    Here since $\beta_i(t)$ varies
     smoothly with respect to the parameter $t$ for any $i\in [m]$, so does $\alpha^I_i(t)$ for any $i\in I$.
     Therefore, the map $\overline{\varphi}_I$ is continuous. And by
     the same argument in section 5 of~\cite{MasIshFuk2010}, we can
     show that $\overline{\varphi}_I$ is a homeomorphism. So we have
     a finite set of charts $\{ \overline{\varphi}_I: U(I)\times [0,1] \slash \Lambda^{-1}(e) \rightarrow \mathbb{C}^I\times
   [0,1]\}_{I\in \Sigma^{(n)}}$ which cover $U(I)\times [0,1] \slash \Lambda^{-1}(e)$.
   Moreover, we can see from Lemma 5.2 in~\cite{MasIshFuk2010} that
   the transition function from such a chart
   indexed by $I\in \Sigma^{(n)}$ to another one indexed by $J\in \Sigma^{(n)}$ is
  $ \overline{\varphi}_J \circ (\overline{\varphi}_I)^{-1} :   \mathbb{C}^I\times [0,1] \rightarrow
     \mathbb{C}^J\times [0,1]$ where
   \begin{equation} \label{Equ:Trans-Func}
    \overline{\varphi}_J \circ (\overline{\varphi}_I)^{-1}
  \left( (z_i)_{i\in I} , t \right)
   = \left( (\prod_{i\in I} z_i^{\langle\alpha^J_j(t),\beta_i(t) \rangle})_{j\in J} , t
   \right),\  \left( (z_i)_{i\in I} , t \right)\in \mathbb{C}^I\times
   [0,1].
   \end{equation}
   This function is smooth since $z_i\neq 0$ for $i \in I\setminus J$ and $\langle\alpha^J_j(t),\beta_i(t)
   \rangle = \delta_{ji} \mathbf{1}$ for $i\in J$. Therefore,
   $\{ \overline{\varphi}_I: U(I)\times [0,1] \slash \Lambda^{-1}(e) \rightarrow \mathbb{C}^I\times
   [0,1]\}_{I\in \Sigma^{(n)}}$ determines a smooth structure on $U(\Sigma)\times [0,1] \slash
   \Lambda^{-1}(e)$.  So the claim is proved. \nn

    Next, we define a map
       \begin{align}
      p: U(\Sigma) \times [0,1] \slash  \Lambda^{-1}(e) \longrightarrow
      [0,1]\ \, \text{where}\ p([(x,t)]) = t.
      \end{align}
  It is clear that $p$ is a smooth map. And for any $t \in [0, 1]$, we have
    $$p^{-1}(t) = U(\Sigma)\times \{ t\}
   \slash  \Lambda_t^{-1}(e) \cong U(\Sigma) \slash \mathrm{Ker}(\lambda_{\beta(t)}) =
    X(\Delta(t)).$$
  Notice that
   the set of charts $\{ \overline{\varphi}_I: U(I)\times [0,1]
   \slash \Lambda^{-1}(e) \rightarrow \mathbb{C}^I\times
   [0,1]\}_{I\in \Sigma^{(n)}}$ defined by~\eqref{Equ:Charts-2} restricted
   to each $p^{-1}(t)\cong X(\Delta(t))$
    give us exactly the set of all normal charts of $X(\Delta(t))$. So
    $p^{-1}(t)$ is an embedding submanifold of $U(\Sigma) \times [0,1] \slash  \Lambda^{-1}(e)$. \nn

   Let $(S^1)^m \subset (\mathbb{C}^*)^m$
   act on $U(\Sigma)\times [0,1] \subset \mathbb{C}^m\times [0,1]$ by the natural action
   \[  g\cdot (x,t) = (g\cdot x, t), \ \forall\,
    g\in (S^1)^m,  \forall\, (x,t) \in \mathbb{C}^m\times [0,1].  \]
   Then by~\eqref{Equ:Bundle-Trivial},
   we have a smooth action of $(S^1)^m \slash  \widehat{\Lambda}_0^{-1}(e) \cong (S^1)^n$
   on the whole $U(\Sigma) \times [0,1] \slash \Lambda^{-1}(e)$, whose restriction
    to each fiber $p^{-1}(t) \cong X(\Delta(t))$ is exactly
     the canonical $(S^1)^n$-action on $X(\Delta(t))$.\nn

      Obviously, $p$ is a smooth proper submersion.
     So Ehresmann's fibration theorem (see~\cite{Ehres50} or~\cite{AbrRobin67}) implies that
    $p: U(\Sigma) \times [0,1] \slash \Lambda^{-1}(e) \rightarrow [0,1]$
    is a locally trivial fiber bundle. Therefore, there exists a diffeomorphism
   from $p^{-1}(0) = X(\Delta(0))$ to $p^{-1}(1) = X(\Delta(1))$.
    But if we want to require the diffeomorphism to be
   equivariant with respect to the canonical $(S^1)^n$-action on $X(\Delta(0))$ and
   $X(\Delta(1))$, we need to refine the proof of Ehresmann's fibration theorem as follows.\nn

   By our previous discussion,
  $\{ U(I)\times [0,1] \slash \Lambda^{-1}(e) \, ;\,I\in \Sigma^{(n)} \}$
  is a finite open cover of $U(\Sigma) \times [0,1] \slash \Lambda^{-1}(e)$
  where each $ U(I)\times [0,1] \slash \Lambda^{-1}(e) $ is
   an invariant open set under the $(S^1)^n$-action.
  Then we can take a $(S^1)^n$-invariant partition of unity
   $$\{ f_I : U(I)\times [0,1] \slash \Lambda^{-1}(e) \rightarrow \R \, ; \, I\in \Sigma \}$$
    subordinate to this open cover, where each $f_I$ is a $(S^1)^n$-invariant function.  \nn

   Next, we take a local
   vector field $Y_I$ on $U(I)\times [0,1] \slash \Lambda^{-1}(e)$ so that
   \[ p_*(Y_I) = \frac{\partial}{\partial t}. \]
   And we set
   \[ \widetilde{Y}_I := \int_{(S^1)^n} Y_I\, d\mu  \]
  where $d\mu$ is a Haar measure on $(S^1)^n$. It is clear that
  \[ p_*(\widetilde{Y}_I) =  \frac{\partial}{\partial t}.\]
   And so we obtain a $(S^1)^n$-invariant vector field $\widetilde{Y}:= \sum_{I\in \Sigma} f_I \widetilde{Y}_I$
   on the whole space $U(\Sigma)\times [0,1] \slash \Lambda^{-1}(e)$. \nn

  Let $\psi_t$ be the flow of $\widetilde{Y}$ on $U(\Sigma)\times [0,1] \slash \Lambda^{-1}(e)$.
  Since $\widetilde{Y}$ is $(S^1)^n$-invariant, $\psi_t$ is also
  $(S^1)^n$-invariant. So
    \[ \psi=\psi_1:   X(\Delta(0)) \cong p^{-1}(0) \longrightarrow   p^{-1}(1) \cong X(\Delta(1)).  \]
   is a $(S^1)^n$-equivariant diffeomorphism. This proves the property (i) for
   $\psi$.\nn

   Furthermore, since $\psi$ is equivariant with respect to the $(S^1)^n$-action,
   Remark~\ref{Rem:Char-SubMfd} implies that $\psi$ will
   map each characteristic submanifold $X(\Delta(0))_i$ in $X(\Delta(0))$ to
   the characteristic submanifold $X(\Delta(1))_i$ in $X(\Delta(1))$ for any $i\in \Sigma^{(1)}$.
   Then by~\eqref{Equ:V(I)}, $\psi$ will map the open subset $V_I(0)$ in $X(\Delta(0))$ to
   $V_I(1)$ in $X(\Delta(1))$ for any $I\in
   \Sigma^{(n)}$. This finishes our proof.
 \end{proof}
  \n

  \begin{rem}
     For a simplicial complex $\Sigma$ and two complete non-singular toric fans $\Delta=(\Sigma,\beta)$
     and $\Delta'=(\Sigma,\beta')$, it is possible that
     $X(\Delta)$ is homeomorphic to
      $X(\Delta')$ while $\Delta$
     and $\Delta'$ can not be connected by any regular deformation. In this situation,
     it is not so clear whether we can find a diffeomorphism from $X(\Delta)$
     to $X(\Delta')$.
  \end{rem}

 \begin{thm} \label{thm:Fan-Deform-Nice}
  For any ominioriented topological toric manifold $X$, there always exists
  a regular deformation of the topological fan of $X$ into a nice topological fan.
   \end{thm}

 \begin{proof}
  Suppose $X$ is a $2n$-dimensional ominioriented topological toric manifold whose topological fan
  is $\Delta(X) = (\Sigma(X), \beta(X))$, where $\Sigma(X)$ and $\beta(X)$
  are defined by~\eqref{Equ:Fan} and~\eqref{Equ:Beta}, respectively.
  We will define a sequence of regular deformations of $\Delta(X)$ below.
  But for the sake of conciseness, we will define each deformation
   in terms of the deformation on the $b_i(X)$ and $c_i(X)$ components. \nn

   Step 1: we deform all $c_i(X)$ to $0$ simultaneously by $(1-t)c_i(X)$,
    $0 \leq  t \leq 1$,  $i\in [m] := \Sigma^{(1)}(X)$.
   Obviously,
      the deformation is regular. We denote the new topological fan obtained from this deformation by
     $\widetilde{\Delta}^{(1)} = (\Sigma(X), \widetilde{\beta}^{(1)})$ where
     $$ \widetilde{\beta}^{(1)}_i = (b_i(X)+ \sqrt{-1} \cdot 0, v_i(X))
     \in \R^n\times \Z^n \subset \mathbb{C}^n\times \Z^n, \ i\in [m].$$
     \n

   Step 2: we deform each $b_i(X)$ in  $\widetilde{\Delta}^{(1)}$ slightly into a
     vector $b'_i(X)$ with rational coordinates.
        Notice that the condition on $\{ b_i(X), i\in [m] \}$ in the definition of
        a complete non-singular topological fan is
        stable under small deformations. So
        we can choose our deformation of $\widetilde{\Delta}^{(1)}$ here to be regular.
        We denote the new topological fan obtained from this deformation by
      $\widetilde{\Delta}^{(2)} = (\Sigma(X), \widetilde{\beta}^{(2)})$ where
       $$ \widetilde{\beta}^{(2)}_i = (b'_i(X)+ \sqrt{-1} \cdot 0, v_i(X))
     \in \Q^n\times \Z^n \subset \mathbb{C}^n\times \Z^n, \ i\in [m].$$

     \n

   Step 3: we can choose a very large positive even integer $N$ so that
   \[ \widehat{b}_i(X) := N\cdot b'_i(X) \in 2\Z^n \; \text{for any}\ i\in [m]. \]
    Obviously $\{ (1-t)b'_i(X) + t\widehat{b}_i(X),\, i\in [m]\}$
    determines a regular deformation of $\widetilde{\Delta}^{(2)}$.
   We denote the new topological fan obtained from this
   deformation by $\widetilde{\Delta}^{(3)} = (\Sigma(X), \widetilde{\beta}^{(3)})$ where
   \[  \widetilde{\beta}^{(3)}_i = (\widehat{b}_i(X)+ \sqrt{-1} \cdot 0, v_i(X))
     \in 2\Z^n\times \Z^n \subset \mathbb{C}^n\times \Z^n, \ i\in [m].  \]
   \nn

  Step 4:  Let $u_i := \widehat{b}_i(X) - v_i(X)\in \Z^n \ \text{for each}\ i\in [m]$.
         Notice in step 3, if we choose the integer $N$ large
         enough, we can assume $\| v_i(X) \|$ is by far smaller than
         $\| \widehat{b}_i(X) \|$. Then the distance between the unit vectors
           $u_i \slash \| u_i \|$ and $ \widehat{b}_i(X)\slash \|  \widehat{b}_i(X)\|$ can be
           made arbitrarily small. It is easy to see that if
            $u_i \slash \| u_i \|$ and $ \widehat{b}_i(X)\slash \|  \widehat{b}_i(X)\|$
           are very close,
          $\{ (1-t)\widehat{b}_i(X) + tu_i, \ i\in [m]
             \}$
         will determine a regular deformation of $\widetilde{\Delta}^{(3)}$.
           We denote the new topological fan obtained from this deformation by
      $\widetilde{\Delta}^{(4)}=(\Sigma(X), \widetilde{\beta}^{(4)})$ where
     \[  \widetilde{\beta}^{(4)}_i = (u_i+ \sqrt{-1} \cdot 0, v_i(X))
     \in \Z^n\times \Z^n \subset \mathbb{C}^n\times \Z^n, \ i\in [m].  \]
        For any $i\in [m]$, we have
        $u_i - v_i(X) = \widehat{b}_i(X) - 2v_i(X) \in 2\Z^n$, so
        $$u_i \equiv v_i(X) \mod 2.$$
        This means that the topological fan $\widetilde{\Delta}^{(4)}$ is nice.
        Note that the integral vector $u_i$ is not necessarily
        primitive.
      By combining the above four steps, we then obtain a regular
        deformation from $\Delta(X)$ to a nice topological fan
        $\widetilde{\Delta}^{(4)}$.
   \end{proof}

    From Lemma~\ref{Lem:Deformation} and Theorem~\ref{thm:Fan-Deform-Nice},
     we immediately get the following.

   \begin{thm} \label{thm:Nice-Fan}
    For any ominioriented topological toric manifold $X$ of dimension $2n$, there
    exists a $(S^1)^n$-equivariant diffeomorphism from $X$ to an ominioriented
   topological toric manifold $\widetilde{X}$ whose topological fan is
   nice.
   \end{thm}

    Moreover, we have the following corollary on the transition
    functions of a topological toric manifold.\nn

  \begin{cor} \label{Cor:Trans-Func}
    Any topological toric manifold $X$ of dimension $2n$ can be covered
    by finitely many
    open charts $\{ \phi_j: V_j \rightarrow \mathbb{C}^n \}_{1\leq j \leq
    r}$ so that each $V_j$ is a $(\mathbb{C}^*)^n$-invariant open subset of $X$
    and, all the transition functions between these charts are Laurent
     monomials of $z_1,\ldots, z_n, \overline{z}_1,\ldots, \overline{z}_n$.
   \end{cor}
   \begin{proof}
   By choosing an ominiorientation on $X$, we can assume $X=X(\Delta)$ where $\Delta=(\Sigma,\beta)$ is a
   complete non-singular topological fan.
    Then by Theorem~\ref{thm:Fan-Deform-Nice}, there exists a diffeomorphism
   $\psi: X \rightarrow \widetilde{X} = X(\widetilde{\Delta})$
   where $\widetilde{\Delta}=(\Sigma, \widetilde{\beta})$ is a nice topological fan. Moreover,
    for any $I\in \Sigma^{(n)}$, let
   $\varphi_I: V_I \rightarrow \mathbb{C}^n$ and
   $\widetilde{\varphi}_I: \widetilde{V}_I \rightarrow \mathbb{C}^n$ be the normal charts of
    $X$ and $\widetilde{X}$, respectively. Then by Lemma~\ref{Lem:Deformation}, we have
     $$\psi(V_I) =\widetilde{V}_I.$$

    Since $(\Sigma, \widetilde{\beta})$ is a nice topological fan,
  the transition functions between the charts
  $\{ \widetilde{\varphi}_I : \widetilde{V}_I \rightarrow \mathbb{C}^n, \
   I\in \Sigma^{(n)}\}$ are all
   Laurent monomials of $z_1,\ldots, z_n, \overline{z}_1,\ldots, \overline{z}_n$.
   Then
   $\{  \widetilde{\varphi}_I\circ\psi :  V_I=\psi^{-1} ( \widetilde{V}_I)\rightarrow \mathbb{C}^n,
   \, I\in \Sigma^{(n)} \}$ are finitely many open charts
   which cover $X$ and clearly satisfy all our requirements.
   \end{proof}
   \n

   \begin{rem}
     For an open chart
    $\widetilde{\varphi}_I\circ\psi :  V_I=\psi^{-1} ( \widetilde{V}_I)\rightarrow \mathbb{C}^n
    $ on $X$ in the proof of Corollary~\ref{Cor:Trans-Func}, the $(\mathbb{C}^*)^n$-action
      in $V_I$ may not be equivariantly homeomorphic to a smooth linear representation of
     $(\mathbb{C}^*)^n$ on $\mathbb{C}^n$ via the map $\widetilde{\varphi}_I\circ\psi$, but
     it does via the map $\varphi_I: V_I \rightarrow \mathbb{C}^n$.
     The reason is that the diffeomorphism $\psi$ we obtain in Lemma~\ref{Lem:Deformation}
    is not necessarily $(\mathbb{C}^*)^n$-equivariant.
    \end{rem}

      There should be some special geometrical properties on a topological toric manifold
   implied by the existence of the kind of atlas
   as described in Corollary~\ref{Cor:Trans-Func}. It is interesting to
    see what these geometric properties are.\nn

  In addition, by Theorem 10.2 in~\cite{MasIshFuk2010} and our preceding discussion, we can easily prove
    the following theorem for quasitoric manifolds.\nn

   \begin{thm} \label{thm:Quasitoric}
    Any $2n$-dimensional quasitoric manifold over a simple convex polytope $P^n$ can be covered by
    finitely many $(S^1)^n$-invariant open charts
    $$\phi_j : V_j \rightarrow \mathbb{C}^n , \,\ 1\leq j \leq r\ \text{ where $r=$
     the number of vertices of $P^n$, } $$
   whose
   transition functions are all Laurent monomials of
   $z_1,\ldots, z_n, \overline{z}_1,\ldots, \overline{z}_n$.
   \end{thm}
   \begin{proof}
     For any $2n$-dimensional quasitoric manifold $M^{2n}$ over a simple convex polytope $P^n$, there
     exists a topological toric manifold $X(\Delta)$ so that
     $M^{2n}$ is equivariantly homeomorphic to $X(\Delta)$ as
     $(S^1)^n$-manifold. In the topological fan $\Delta = (\Sigma, \beta)$, the simplical complex
     $\Sigma$ is the boundary of the dual simplicial polytope of $P^n$.
     So the number of $(n-1)$-simplices in $\Sigma$ equals the number of vertices of $P^n$.
       Moreover, we can require the topological fan $\Delta$ is nice by
     Theorem~\ref{thm:Nice-Fan}. Let $\psi$ be an $(S^1)^n$-equivariant
     homeomorphism from $M^{2n}$ to $X(\Delta)$.
     Then we can use $\psi$ to pull back the atlas on $X(\Delta)$ described
     in Corollary~\ref{Cor:Trans-Func} to an atlas on $M^{2n}$, which
     will satisfy all our requirements.
   \end{proof}
  \n
   \begin{rem}
       For a $2n$-dimensional quasitoric manifold $M^{2n}$,
       the set of $(S^1)^n$-invariant open charts described in Theorem~\ref{thm:Quasitoric}
       will determine
       a smooth structure on $M^{2n}$. But it is not clear whether
       this smooth structure should agree with the original smooth
       structure of $M^{2n}$ (see the remark after the Theorem 10.2
       in~\cite{MasIshFuk2010}).\\
   \end{rem}

   \section{Nice topological toric manifolds and
   real algebraic representation of $(\mathbb{C}^*)^n$}

  \begin{defi}[Nice Topological Toric Manifold]
   A topological toric manifold $X$ is called \textit{nice} if there exists
   an ominiorientation on $X$
   so that the associated topological fan is nice.
  \end{defi}

  \begin{defi}[Real Algebraic Linear Representation]
    A faithful smooth linear representation $\rho$
    of $(\mathbb{C}^*)^n$ on $\mathbb{C}^n$ is called \textit{real algebraic} if
    $$\rho(z_1,\cdots, z_n) = (h_1(z_1,\ldots,z_n,\overline{z}_1,\ldots, \overline{z}_n),
     \cdots, h_{n}(z_1,\ldots, z_n, \overline{z}_1,\ldots,
     \overline{z}_n)) \in (\mathbb{C}^*)^n$$
    where $h_j(z_1,\ldots,z_n,\overline{z}_1,\ldots, \overline{z}_n)$
    is a Laurent monomial of $z_1,\ldots, z_n, \overline{z}_1,\ldots,
    \overline{z}_n$ for each $1\leq j \leq n$.
  \end{defi}

   By the discussion in section 2 of~\cite{MasIshFuk2010}, any faithful smooth
   complex $n$-dimensional representation $V$ of $(\mathbb{C}^*)^n$ can be written as
    \[ V = V(\bigoplus^n_{i=1}\chi^{\alpha_i}), \  \text{where}\
     \chi^{\alpha_i} \in  \mathrm{Hom}((\mathbb{C}^*)^n, \mathbb{C}^*).
     \]
     We can identify $\mathrm{Hom}((\mathbb{C}^*)^n, \mathbb{C}^*)$ with
    the row vector space of $\mathbb{C}^n\times \Z^n$ and write
     \[  \alpha_i = (x_i + \sqrt{-1} y_i, u_i) \in \mathbb{C}^n\times \Z^n \]
     where $x_i,y_i \in \R^n$ and $u_i \in \Z^n$ are row
     vectors.
    By Lemma 2.1 in~\cite{MasIshFuk2010}, it is easy to see the following.
  \[
     \text{$V(\bigoplus^n_{i=1}\chi^{\alpha_i})$ is
     real algebraic  $\Longleftrightarrow$ $y_i = 0$ and $x_i\in \Z^n$ with $x_i\equiv u_i
     \ \mathrm{mod}\ 2$ for all $1\leq i\leq n$} \]

   Suppose $\{ \beta_i \}_{1\leq i \leq n}$ is a \textit{dual set} of
   $\{ \alpha_i \}_{1\leq i \leq n}$, i.e.
   $\lambda_{\beta_i} \in \mathrm{Hom}(\mathbb{C}^*,
   (\mathbb{C}^*))^n$ so that
   \[ \lambda_{\beta_j}\circ \chi^{\alpha_i} = \delta_{ij}
   id_{(\mathbb{C}^*)^n},\ \, \chi^{\alpha_i}\circ\lambda_{\beta_j} = \delta_{ij}
   id_{\mathbb{C}^*}\ \, \text{for all}\ 1\leq i,j\leq n.
   \]
   If we write $\beta_i = (b_i + \sqrt{-1}c_i, v_i) \in \mathbb{C}^n\times \Z^n$,
    it is easy to see the following.
   \[
     \text{$V(\bigoplus^n_{i=1}\chi^{\alpha_i})$ is
     real algebraic  $\Longleftrightarrow$ $c_i = 0$ and $b_i\in \Z^n$ with $b_i \equiv v_i
     \ \mathrm{mod}\ 2$ for all $1\leq i\leq n$} \]
    \n

    Then by our discussion at the end of section 2,
     for any nice topological fan $\Delta = (\Sigma, \beta)$ of dimension $n$,
   the corresponding nice topological toric manifold $X(\Delta)$ can be covered by finitely many
   $(\mathbb{C}^*)^n$-invariant open charts
   $\{ \varphi_I : V_I \rightarrow \mathbb{C}^n\}_{I \in \Sigma^{(n)}}$
       so that for each $\varphi_I : V_I \rightarrow \mathbb{C}^n$, the $(\mathbb{C}^*)^n$-action in $V_I$ is
       equivariant homeomorphic to a real algebraic linear representation
       of $(\mathbb{C}^*)^n$ in $\mathbb{C}^n$ via the map $\varphi_I$.\nn

    Similar to the description of toric manifolds in Theorem~\ref{thm:Toric-Trans-Func},
    we have the following
   description of nice topological toric manifolds in terms of transition functions
   of charts up to (weakly) equivariant diffeomorphism. \nn

     \begin{thm}\label{thm:Top-Toric-Trans-Func}
        Suppose $M^{2n}$ is a $2n$-dimensional closed smooth manifold with an
       effective smooth action of $(\mathbb{C}^*)^n$ having an open dense orbit.
       Then $M^{2n}$ is (weakly) equivariantly
       diffeomorphic to a nice topological toric manifold if and only if
        $M^{2n}$ can be covered by finitely many $(\mathbb{C}^*)^n$-invariant open charts
       $\{ \phi_j: V_j \rightarrow \mathbb{C}^n\}_{1\leq j \leq r}$
       so that all the transition functions between these charts are Laurent monomials of
        $z_1,\ldots, z_n, \overline{z}_1,\ldots, \overline{z}_n$
        and, for at least one chart $\phi_j: V_j \rightarrow \mathbb{C}^n$,
        the $(\mathbb{C}^*)^n$-action in $V_j$ is equivariant homeomorphic to a real algebraic
        linear representation of $(\mathbb{C}^*)^n$ on $\mathbb{C}^n$ via the map $\phi_j$.
   \end{thm}
   \begin{proof}
    Suppose there exists a (weakly) equivariant diffeomorphism $\psi$ from
    $M^{2n}$ to a nice topological toric manifold $X= X(\Delta)$ where
    $\Delta=(\Sigma,\beta)$ is a nice topological fan. By our previous discussion,
    $X(\Delta)$ can be covered by finitely many
   $(\mathbb{C}^*)^n$-invariant open charts
   $\{ \varphi_I : V_I \rightarrow \mathbb{C}^n\}_{I \in \Sigma^{(n)}}$
       so that
        the $(\mathbb{C}^*)^n$-action in each $V_I$ is
       equivariant homeomorphic to a real algebraic linear representation
       of $(\mathbb{C}^*)^n$ in $\mathbb{C}^n$ via the map $\varphi_I$.
    It is clear that each $\psi^{-1}(V_I)$ is a $(\mathbb{C}^*)^n$-invariant open subset of $M^{2n}$.
   So we can cover $M^{2n}$ by finitely many
    $(\mathbb{C}^*)^n$-invariant open charts
    $$\{ \phi_I \circ \psi: \psi^{-1}(V_I) \rightarrow \mathbb{C}^n \}_{I \in \Sigma^{(n)}}$$
   whose transition functions are all Laurent monomials of
        $z_1,\ldots, z_n, \overline{z}_1,\ldots, \overline{z}_n$.
       Moreover, since $\psi$ is a
   (weakly) equivariant diffeomorphism, so
   the $(\mathbb{C}^*)^n$-action in each $\psi^{-1}(V_I) \subset M^{2n}$ is
   (weakly) equivariantly differomorphic to a real algebraic linear representation
   of $(\mathbb{C}^*)^n$ in $\mathbb{C}^n$ via $\phi_I \circ \psi$. This proves the ``only if''
   part of the theorem.\nn

     The ``if'' part of the theorem is very similar to the proof of
     Theorem~\ref{thm:Toric-Trans-Func}, so we leave it as an exercise to the reader.
   \end{proof}

    \nn

    \nn

  \section*{Acknowledgements}
     Theorem~\ref{thm:Toric-Trans-Func} in this paper was told to the
     author by Mikiya Masuda and has been
     known to him and perhaps some other people for quite a while.
     In addition, the author is indebted to H.~Ishida for
     clarifying many details in the proof of Lemma~\ref{Lem:Deformation} and
    some important comments.\\


\begin{thebibliography}{99}

   \bibitem{MasIshFuk2010}
  H.~Ishida, Y.~Fukukawa and M.~Masuda, \textit{Topological toric
  manifold}, arXiv:1012.1786

\bibitem{DaJan91}  M.~W.~Davis and T.~Januszkiewicz, \textit{Convex polytopes,
Coxeter orbifolds and torus actions}, Duke Math. J. \textbf{62}
(1991), no.\textbf{2}, 417--451.

\bibitem{BP02} V. M. Buchstaber and T.E. Panov,
{\em Torus actions and their applications in topology and
 combinatorics}, University Lecture Series, 24.
 American Mathematical Society, Providence, RI, 2002.

\bibitem{Danilov}
  V.~I.~Danilov, \textit{The geometry of toric varieties},
  Uspekhi Mat. Nauk 33 (1978), no.~\textbf{2}, 85--134; English translation
  in: Russian Math. Surveys 33 (1978), no.~\textbf{2}, 97--154.

\bibitem{Oda}
  T.~Oda, Convex bodies and algebraic geometry.
    An introduction to the theory of toric varieties. Springer-Verlag,
  Berlin, 1988.

\bibitem{Ehres50}
  C.~Ehresmann, \textit{Les connexions infinit\'esimales dans un espace
  fibr\'e diff\'erentiable}, Colloque de Topologie, Bruxelles (1950),
  29--55.

\bibitem{AbrRobin67}
 R.~Abraham and J.~Robin, \textit{Transversal mappings and flows}, W.~A.~Benjamin, Inc., 1967.

\end{thebibliography}
\end{document}